\documentclass{amsart}
\usepackage{amssymb,amstext,amsmath,amscd,amsthm,amsfonts,enumerate,graphicx,latexsym,stmaryrd}
\usepackage[usenames]{color}
\usepackage[all]{xy}
\newtheorem{thm}{Theorem}[section]
\newtheorem{lem}[thm]{Lemma}
\newtheorem{prop}[thm]{Proposition}
\newtheorem{cor}[thm]{Corollary}
\theoremstyle{definition}
\newtheorem{dfn}[thm]{Definition}

\newtheorem{rem}[thm]{Remark}

\newtheorem{ex}[thm]{Example}

\theoremstyle{remark}

\numberwithin{equation}{thm}
\def\cx{\operatorname{cx}}
\def\depth{\operatorname{depth}}
\def\Ext{\operatorname{Ext}}
\def\G{\mathcal{G}}
\def\gcx{\operatorname{gcx}}
\def\gdim{\operatorname{Gdim}}
\def\ge{\geqslant}
\def\Hom{\operatorname{Hom}}
\def\im{\operatorname{Im}}
\def\le{\leqslant}
\def\m{\mathfrak{m}}
\def\mod{\operatorname{mod}}
\def\modo{\operatorname{mod_0}}
\def\ured{\operatorname{\mathsf{{}^\ast\!red}}}
\def\N{\mathbb{N}}
\def\op{\mathrm{op}}
\def\pd{\operatorname{pd}}
\def\px{\operatorname{px}}
\def\red{\operatorname{\mathsf{red}}}

\def\syz{\mathrm{\Omega}}
\def\Tor{\operatorname{Tor}}
\def\tr{\operatorname{Tr}}
\def\X{\mathcal{X}}
\begin{document}
\allowdisplaybreaks
\title{On reducing homological dimensions over noetherian rings}
\author{Tokuji Araya}
\address{Department of Applied Science, Faculty of Science, Okayama University of Science, Ridaicho, Kitaku, Okayama 700-0005, Japan}
\email{araya@das.ous.ac.jp}
\author{Ryo Takahashi}
\address{Graduate School of Mathematics, Nagoya University, Furocho, Chikusaku, Nagoya 464-8602, Japan}
\email{takahashi@math.nagoya-u.ac.jp}
\urladdr{https://www.math.nagoya-u.ac.jp/~takahashi/}
\thanks{Ryo Takahashi was partly supported by JSPS Grant-in-Aid for Scientific Research 19K03443}
\begin{abstract}
Let $\Lambda$ be a left and right noetherian ring.
First, for $m,n\in\N\cup\{\infty\}$, we give equivalent conditions for a given $\Lambda$-module to be $n$-torsionfree and have $m$-torsionfree transpose.
Using them, we investigate totally reflexive modules and reducing Gorenstein dimension.
Next, we introduce homological invariants for $\Lambda$-modules which we call upper reducing projective and Gorenstein dimensions.
We provide an inequality of upper reducing projective dimension and complexity when $\Lambda$ is commutative and local.
Using it, we consider how upper reducing projective dimension relates to reducing projective dimension, and the complete intersection and AB properties of a commutative noetherian local ring.
\end{abstract}
\maketitle
\section{Introduction}

Araya and Celikbas \cite{AC} introduce homological invariants for modules called reducing homological dimensions.
In this paper, we study two of them: reducing projective dimension and reducing Gorenstein dimension.

Let $\Lambda$ be a left and right noetherian ring.
For nonnegative integers $m,n$ we say that a finitely generated right $\Lambda$-module $M$ is {\em $(m,n)$-torsionfree} if $\Ext_\Lambda^i(M,\Lambda)=\Ext_{\Lambda^\op}^j(\tr M,\Lambda)=0$ for all $1\le i\le m$ and $1\le j\le n$.
This is none other than an $n$-torsionfree module whose transpose is $m$-torsionfree.
The notion of an $(m,n)$-torsionfree module is naturally extended to the case where $m=\infty$ or $n=\infty$.
We provide criteria for the $(m,n)$-torsionfree property, part of which is the following.

\begin{thm}\label{14}
Let $\Lambda$ be a left and right noetherian ring.
Let $m,n\in\N\cup\{\infty\}$.
A finitely generated right $\Lambda$-module $M$ is $(m,n)$-torsionfree if and only if it admits an exact sequence
$$
G_{m+1}\to G_m\to\cdots\to G_0\xrightarrow{\partial}G_{-1}\to\cdots\to G_{-n}
$$
of $(m,n)$-torsionfree $\Lambda$-modules whose $\Lambda$-dual is also exact such that $M$ is isomorphic to the image of $\partial$.
\end{thm}

Applying this theorem, we obtain the following corollary.
The first statement is included in a theorem of Sather-Wagstaff, Sharif and White \cite{SSW} if $\Lambda$ is commutative.
The second statement includes a theorem of Araya and Celikbas \cite{AC}, which shows the same assertion under the additional assumption that $\Lambda$ is commutative and local.
Furthemore, our proof highly simplifies theirs.

\begin{cor}\label{15}
Let $\Lambda$ be a left and right noetherian ring.
Let $M$ be a finitely generated right $\Lambda$-module.
\begin{enumerate}[\rm(1)]
\item
The module $M$ is totally reflexive if and only if there is an exact sequence of totally reflexive $\Lambda$-modules
$$
\cdots\to G_1\to G_0\xrightarrow{\partial}G_{-1}\to G_{-2}\to\cdots
$$
whose $\Lambda$-dual is also exact such that $M$ is isomorphic to the image of $\partial$.
\item
If $M$ has finite reducing Gorenstein dimension, then $\gdim_\Lambda M=\sup\{i\in\N\mid\Ext_\Lambda^i(M,\Lambda)=0\}$.
\end{enumerate}
\end{cor}

Next, modifying the definition of reducing projective dimension, we introduce upper reducing projective dimension, which is always at least reducing projective dimension.
We prove the following theorem, which relates it to complexity.

\begin{thm}\label{16}
Let $R$ be a commutative noetherian local ring.
Let $M$ be a finitely generated $R$-module.
Then the upper reducing projective dimension of $M$ is more than or equal to the complexity of $M$.
If $M$ has reducible complexity (this is the case when $R$ is a complete intersection), then both are equal.
\end{thm}

Applying this theorem, we obtain the following corollary, providing some relationships of upper reducing projective dimension with reducing projective dimension, complete intersections and AB rings.

\begin{cor}\label{17}
Let $R$ be a commutative noetherian local ring.
Consider the following conditions.
\begin{enumerate}[\rm(1)]
\item
The local ring $R$ is a complete intersection.
\item
Every finitely generated $R$-module has reducible complexity.
\item
Every finitely generated $R$-module has finite upper reducing projective dimension.
\item
The residue field of $R$ has finite upper reducing projective dimension.
\item
Every finitely generated $R$-module has finite reducing projective dimension.
\item
The local ring $R$ is AB.
\end{enumerate}
Then the implications {\rm(1)} $\Leftrightarrow$ {\rm(2)} $\Leftrightarrow$ {\rm(3)} $\Leftrightarrow$ {\rm(4)} $\Rightarrow$ {\rm(5)} $\Rightarrow$ {\rm(6)} hold.
\end{cor}

The paper is organized as follows.
In Section \ref{12}, we deal with reducing Gorenstein dimension, and for this we investigate $(m,n)$-torsionfree modules.
Theorem \ref{14} and Corollary \ref{15} are proved in this section.
In Section \ref{11}, we make the definitions of upper reducing projective dimension and upper reducing Gorenstein dimension, and explore the former.
Proofs of Theorem \ref{16} and Corollary \ref{17} are stated in this section.
In Section \ref{13}, we study upper reducing Gorenstein dimension, and obtain a similar inequality as in Theorem \ref{16}.

\section{$(m,n)$-torsionfree modules and reducing Gorenstein dimension}\label{12}

In this section, we give a characterization of the $(m,n)$-torsionfree modules in terms of the existence of certain exact sequences.
As applications, we obtain a characterization of the totally reflexive modules and a formula of Gorenstein dimension for modules of finite reducing Gorenstein dimension.

Let $\Lambda$ be a (left and right) noetherian ring.
Denote by $\mod\Lambda$ the category of finitely generated (right) $\Lambda$-modules.
For each $M\in\mod\Lambda$ we denote by $\syz M$ and $\tr M$ the {\em (first) syzygy} and the {\em (Auslander) transpose} of $M$ respectively, for whose details we refer the reader to \cite{AB}; see also \cite[\S1]{radius}.

\begin{dfn}
For $m,n\in\N\cup\{\infty\}$ we denote by $\G_{mn}$ the full subcategory of $\mod\Lambda$ consisting of $\Lambda$-modules $M$ such that $\Ext_\Lambda^i(M,\Lambda)=0$ for all $1\le i\le m$ and $\Ext_{\Lambda^\op}^j(\tr M,\Lambda)=0$ for all $1\le j\le n$.
When $m=\infty$ (resp. $n=\infty$), the vanishing condition reads $\Ext_\Lambda^i(M,\Lambda)=0$ for all $i\ge1$ (resp. $\Ext_{\Lambda^\op}^j(\tr M,\Lambda)=0$ for all $j\ge1$).
Note that $\G_{00}=\mod\Lambda$.
We call the modules belonging to $\G_{mn}$ the {\em $(m,n)$-torsionfree} modules.
The $(0,n)$-torsionfree modules are nothing but the {\em $n$-torsionfree} modules in the sense of Auslander and Bridger \cite{AB}.
The $(\infty,\infty)$-torsionfree modules are none other than the {\em totally reflexive} (or {\em Gorenstein projective}) modules in the sense of Avramov and Martsinkovsky \cite{AM} (and Enochs and Jenda \cite{EJ}).
There are inclusions $\G_{m\infty}\subseteq\cdots\subseteq\G_{m,n+1}\subseteq\G_{mn}\subseteq\cdots\subseteq\G_{m0}$ and $\G_{\infty n}\subseteq\cdots\subseteq\G_{m+1,n}\subseteq\G_{mn}\subseteq\cdots\subseteq\G_{0n}$, and there is an equality $\G_{mn}=\G_{m0}\cap\G_{0n}$.
\end{dfn}

It is the main subject of \cite{AB} to investigate the $(m,n)$-torsionfree modules; a lot of deep results on those modules are obtained there.
Also, the subcategories $\G_{mn}$ are closely related to one another, whose stable categories involve equivalences and dualities given by $\syz$ and $\tr$; see \cite[Proposition 1.1.1]{I}.

Put $(-)^\ast=\Hom_\Lambda(-,\Lambda)$.
The main result of this section is the following theorem.

\begin{thm}\label{3}
Let $\Lambda$ be a noetherian ring.
Let $M$ be a finitely generated $\Lambda$-module.
The following are equivalent for $m,n\in\N\cup\{\infty\}$.
\begin{enumerate}[\rm(1)]
\item
The $\Lambda$-module $M$ belongs to $\G_{mn}$, that is to say, $M$ is $(m,n)$-torsionfree.
\item
There exists an exact sequence of finitely generated projective $\Lambda$-modules $P_{m+1}\to P_m\to\cdots\to P_0\xrightarrow{\partial}P_{-1}\to\cdots\to P_{-n}$ whose $\Lambda$-dual is also exact and satisfies $\im\partial\cong M$.
\item
There exists an exact sequence $\sigma:G_{m+1}\to G_m\to\cdots\to G_0\xrightarrow{\partial}G_{-1}\to\cdots\to G_{-n}$ of $\Lambda$-modules such that $G_i\in\G_{mn}$ for all $-n\le i\le m+1$, $\im\partial\cong M$ and $\sigma^\ast$ is also exact.
\item
There is an exact sequence $\sigma:G_{m+1}\to G_m\to\cdots\to G_0\xrightarrow{\partial}G_{-1}\to\cdots\to G_{-n}$ of $\Lambda$-modules such that $G_i\in\G_{m0}$ for all $0\le i\le m+1$, $G_i\in\G_{0n}$ for all $-n\le i\le-1$, $\im\partial\cong M$ and $\sigma^\ast$ is also exact.
\end{enumerate}
\end{thm}

\begin{proof}
(1) $\Rightarrow$ (2):
Taking a projective resolution of $M$, we get an exact sequence $\alpha:P_{m+1}\to P_m\to\cdots\to P_0\to M\to0$ with $P_i$ projective for $0\le i\le m+1$.
As $\Ext_\Lambda^i(M,\Lambda)=0$ for $1\le i\le m$, the dual sequence $0\to M^\ast\to P_0^\ast\to\cdots\to P_{m+1}^\ast$ is exact.
By \cite[Theorem (2.17)]{AB} there is an exact sequence $\beta:0\to M\to P_{-1}\to\cdots\to P_{-n}$ with $P_i$ projective for $-n\le i\le-1$ such that the dual sequence $P_{-n}^\ast\to\cdots\to P_{-1}^\ast\to M^\ast\to0$ is exact.
Splicing $\alpha$ and $\beta$, we obtain an exact sequence as in (2).

(2) $\Rightarrow$ (3):
The implication is evident since $\G_{mn}$ contains the finitely generated projective $\Lambda$-modules.

(3) $\Rightarrow$ (4):
Since $\G_{mn}$ is contained in both $\G_{m0}$ and $\G_{0n}$, the implication holds.

(4) $\Rightarrow$ (1):
As $\G_{mn}=\G_{m0}\cap\G_{0n}$, it suffices to prove the following two statements.

(a) Suppose that there is an exact sequence $\sigma:G_{m+1}\to G_m\to\cdots\to G_0\to M\to0$ such that $G_i\in\G_{m0}$ for all $0\le i\le m+1$ and $\sigma^\ast$ is also exact.
Then $M\in\G_{m0}$.

(b) Suppose that there is an exact sequence $\sigma:0\to M\to G_{-1}\to\cdots\to G_{-n}$ such that $G_i\in\G_{0n}$ for all $-n\le i\le-1$ and $\sigma^\ast$ is also exact.
Then $M\in\G_{0n}$.

First, we prove (a) by induction on $m\in\N$; then the case $m=\infty$ also follows.
There is nothing to prove when $m=0$.
Assume $m\ge1$.
Letting $N$ be the image of the map $G_1\to G_0$, we get exact sequences $G_{m+1}\to\cdots\to G_1\to N\to0$ and $0\to N\to G_0\to M\to0$.
As $m\ge1$, we have exact sequences
$$
{\rm(i)}\ G_2\to G_1\to N\to0,\qquad
{\rm(ii)}\ G_0^\ast\to G_1^\ast\to G_2^\ast,\qquad
{\rm(iii)}\ 0\to N^\ast\to G_1^\ast\to G_2^\ast,
$$
where (iii) is induced from (i).
Using the exactness of (ii) and (iii), we see that the map $G_0^\ast\to N^\ast$ is surjective, which together with the containment $G_0\in\G_{10}$ implies $\Ext_\Lambda^1(M,\Lambda)=0$.
Splicing (iii) with the exact sequence $G_1^\ast\to G_2^\ast\to\cdots\to G_{m+1}^\ast$, we get an exact sequence $0\to N^\ast\to G_1^\ast\to\cdots\to G_{m+1}^\ast$.
The induction hypothesis implies $N\in\G_{m-1,0}$.
An exact sequence $\Ext_\Lambda^{i-1}(N,\Lambda)\to\Ext_\Lambda^i(M,\Lambda)\to\Ext_\Lambda^i(G_0,\Lambda)$ is induced for each $i$.
As $G_0\in\G_{m0}$ and $N\in\G_{m-1,0}$, we observe that $\Ext_\Lambda^i(M,\Lambda)=0$ for all $2\le i\le m$.
We now obtain $M\in\G_{m0}$, as desired.

Next, we prove (b) by induction on $n\in\N$.
There is nothing to show in the case $n=0$, so let $n\ge1$.
Then $G_{-1}$ exists, and this module belongs to $\G_{01}$.
It is observed from \cite[Proposition (2.6)(a)]{AB} and \cite[Lemma 3.4]{EG} that $\G_{01}$ is closed under submodules.
The injection $M\to G_{-1}$ shows $M\in\G_{01}$.
Thus the case $n=1$ is done, and we assume $n\ge2$ from now on.
Letting $N$ be the image of the map $G_{-1}\to G_{-2}$, we have exact sequences $0\to M\to G_{-1}\to N\to0$ and $0\to N\to G_{-2}\to\cdots\to G_{-n}$.
The exact sequences $G_{-2}^\ast\to G_{-1}^\ast\to M^\ast$ and $0\to N^\ast\to G_{-1}^\ast\to M^\ast$ show that the map $G_{-2}^\ast\to N^\ast$ is surjective, and we see that the sequence $G_{-n}^\ast\to\cdots\to G_{-2}^\ast\to N^\ast\to0$ is exact.
By the induction hypothesis we obtain $N\in\G_{0,n-1}$.
An exact sequence $0\to N^\ast\to G_{-1}^\ast\to M^\ast\xrightarrow{\delta}\tr N\to\tr G_{-1}\to\tr M\to0$ (up to free summands) is induced; see \cite[Lemma (3.9)]{AB}.
The surjectivity of the map $G_{-1}^\ast\to M^\ast$ implies that the map $\delta$ is zero, which yields an exact sequence $0\to\tr N\to\tr G_{-1}\to\tr M\to0$.
This induces an exact sequence $\Ext_{\Lambda^\op}^{i-1}(\tr N,\Lambda)\to\Ext_{\Lambda^\op}^i(\tr M,\Lambda)\to\Ext_{\Lambda^\op}^i(\tr G_{-1},\Lambda)$ for each $i$.
Since $G_{-1}\in\G_{0n}$ and $N\in\G_{0,n-1}$, we observe that $\Ext_{\Lambda^\op}^i(\tr M,\Lambda)=0$ for all $2\le i\le n$.
It follows that $M\in\G_{0n}$.
\end{proof}

Letting $m=n=\infty$ in Theorem \ref{3} immediately yields the following result.
This is a special case of the theorem of Sather-Wagstaff, Sharif and White \cite[Theorem A]{SSW} when $\Lambda$ is commutative.

\begin{cor}
Let $\Lambda$ be a noetherian ring.
A finitely generated $\Lambda$-module $M$ is totally reflexive if and only if there exists an exact sequence of totally reflexive $\Lambda$-modules $\cdots\to G_1\to G_0\xrightarrow{\partial}G_{-1}\to G_{-2}\to\cdots$ whose $\Lambda$-dual is also exact and satisfies $\im\partial\cong M$.
\end{cor}

Here we recall the definition of Gorenstein dimension, which has been introduced in \cite{AB}.

\begin{dfn}
Let $\Lambda$ be a noetherian ring and $M$ a finitely generated $\Lambda$-module.
The {\em Gorenstein dimension} ({\em G-dimension} for short) of $M$, which is denoted by $\gdim_\Lambda M$, is defined as the infimum of integers $n\ge0$ such that there exists an exact sequence $0\to G_n\to G_{n-1}\to\cdots\to G_1\to G_0\to M\to0$ of $\Lambda$-modules with $G_i$ being totally reflexive for all $0\le i\le n$.
\end{dfn}

The {\em reducing homological dimensions} are defined in \cite[Definition 2.1]{AC}.
In this paper we deal with two of them.

\begin{dfn}\label{18}
Let $\Lambda$ be a noetherian ring.
Let $M$ be a finitely generated $\Lambda$-module.
We define the {\em reducing projective dimension} $\red_\Lambda^{\pd}M$ and the {\em reducing Gorenstein dimension} $\red_\Lambda^{\gdim}M$ of $M$ as follows; let $\varpi\in\{\pd,\gdim\}$.
\begin{enumerate}[(1)]
\item
One has $\varpi_\Lambda(M)<\infty$ if and only if $\red_\Lambda^\varpi M=0$.
\item
Assume $\varpi_\Lambda(M)=\infty$.
If there exist exact sequences $\{0\to M_{i-1}^{\oplus a_i}\to M_i\to\syz^{n_i}M_{i-1}^{\oplus b_i}\to0\}_{i=1}^s$ in $\mod\Lambda$ with $n_i\ge0$, $s,a_i,b_i>0$, $M_0=M$ and $\varpi_\Lambda(M_s)<\infty$, then $\red_\Lambda^\varpi M$ is the minimum of such integers $s$.
\item
If no such exact sequences as in (2) exist, then $\red_\Lambda^\varpi M=\infty$.
\end{enumerate}
\end{dfn}

\begin{rem}\label{19}
In the original definition of reducing homological dimension given in \cite{AC}, the condition $n_i\ge0$ appearing in Definition \ref{18}(2) is replaced with $n_i>0$.
According to the definition of reducible complexity (see Definition \ref{25}(2) stated later), it is more natural to require $n_i\ge0$ rather than $n_i>0$ (this corresponds to the condition $n\ge0$ in Definition \ref{25}(2)(b)), and so we adopt this modification.
Note that reducing homological dimension in our sense is always smaller than or equal to reducing homological dimension in the sense of \cite{AC}.
\end{rem}

Theorems of Auslander, Buchsbaum and Serre and of Auslander and Bridger are stated in terms of reducing projective and Gorenstein dimensions as follows.

\begin{rem}\label{8}
Let $(R,\m,k)$ be a commutative noetherian local ring.
One has the following equivalences.
\begin{enumerate}[(1)]
\item
$\text{$R$ is regular}\iff
\text{$\red_R^{\pd}M=0$ for all $M\in\mod R$}\iff
\text{$\red_R^{\pd}k=0$.}$
\item
$\text{$R$ is Gorenstein}\iff
\text{$\red_R^{\gdim}M=0$ for all $M\in\mod R$}\iff
\text{$\red_R^{\gdim}k=0$.}$
\end{enumerate}
\end{rem}

We establish a lemma which is a consequence of the definitions of reducing homological dimensions.

\begin{lem}\label{redseq}
Let $\Lambda$ be a noetherian ring, and $M$ a finitely generated $\Lambda$-module.
Let $\varpi\in\{\pd,\gdim\}$, and let $r$ be a positive integer.
If one has $\red_\Lambda^\varpi M\leq r$, then there exist exact sequences $\{0\to M_{i-1}^{\oplus a_i}\to M_i\to\syz^{n_i}M_{i-1}^{\oplus b_i}\to0\}_{i=1}^r$ in $\mod\Lambda$ such that $n_i\ge0$, $a_i,b_i>0$, $M_0=M$ and $\varpi_\Lambda(M_r)<\infty$.
\end{lem}

\begin{proof}
Let $s=\red_\Lambda^\varpi M\le r$.
When $s=0$, we have $\varpi_\Lambda(M)<\infty$, and then we set $M_0=M$.
When $s>0$, there exist exact sequences $\{0\to M_{i-1}^{\oplus a_i}\to M_i\to\syz^{n_i}M_{i-1}^{\oplus b_i}\to0\}_{i=1}^s$ in $\mod\Lambda$ with $n_i\ge0$, $a_i,b_i>0$, $M_0=M$ and $\varpi_\Lambda(M_s)<\infty$.
For each $s+1\le i\le r$, let $a_i,b_i$ be any positive integers (say $a_i=b_i=1$), let $n_i$ be any nonnegative integer (say $n_i=0$), set $M_i=M_{i-1}^{\oplus a_i}\oplus\syz^{n_i}M_{i-1}^{\oplus b_i}$, and take the split exact sequence $0\to M_{i-1}^{\oplus a_i}\to M_i\to\syz^{n_i}M_{i-1}^{\oplus b_i}\to0$.
We obtain exact sequences $\{0\to M_{i-1}^{\oplus a_i}\to M_i\to\syz^{n_i}M_{i-1}^{\oplus b_i}\to0\}_{i=1}^r$ in $\mod\Lambda$ with $n_i\ge0$, $a_i,b_i>0$, $M_0=M$ and $\varpi_\Lambda(M_r)<\infty$.
\end{proof}

The result below includes a theorem of Araya and Celikbas \cite[Theorem 1.3]{AC}; they prove it under the additional assumption that $\Lambda$ is both commutative and local\footnote{By Remark \ref{19}, finiteness of reducing Gorenstein dimension in the sense of \cite{AC} implies finiteness of reducing Gorenstein dimension in our sense. Thus, even when $\Lambda$ is both commutative and local, Corollary \ref{20} is stronger than \cite[Theorem 1.3]{AC}.}.
It is also worth mentioning that our proof is much simpler than that of \cite[Theorem 1.3]{AC}.

\begin{cor}\label{20}
Let $\Lambda$ be a noetherian ring.
Let $M$ be a finitely generated $\Lambda$-module of finite reducing Gorenstein dimension.
Then $\gdim_\Lambda M=\sup\{i\in\N\mid\Ext_\Lambda^i(M,\Lambda)\ne0\}$.
\end{cor}

\begin{proof}
The equality holds if $\gdim_\Lambda M<\infty$ by \cite[Page 95]{AB}.
It suffices to show that $\gdim_\Lambda M<\infty$ if $s:=\sup\{i\in\N\mid\Ext_\Lambda^i(M,\Lambda)\ne0\}$ is finite.
Put $r=\red_\Lambda^{\gdim}M<\infty$.
If $r=0$, then $\gdim_\Lambda M<\infty$ and we are done.
We assume $r>0$.
By Lemma \ref{redseq} there exist exact sequences $\{0\to M_{i-1}^{\oplus a_i}\to M_i\to\syz^{n_i}M_{i-1}^{\oplus b_i}\to0\}_{i=1}^r$ in $\mod\Lambda$ with $n_i\ge0$, $a_i,b_i>0$, $M_0=M$ and $\gdim_\Lambda M_r<\infty$.
Setting $a=a_1$, $b=b_1$, $n=n_1$ and $N=M_1$, we have an exact sequence
$$
0\to M^{\oplus a}\to N\to\syz^nM^{\oplus b}\to0
$$
and $\red_\Lambda^{\gdim}N\le r-1$.
This exact sequence shows that $\sup\{i\in\N\mid\Ext_\Lambda^i(N,\Lambda)\ne0\}\le s<\infty$.
The induction hypothesis implies that $N$ has finite Gorenstein dimension.
To show that $M$ has finite Gorenstein dimension, replacing $M$ and $N$ with their $s$th syzygies, we may assume that $M\in\G_{\infty0}$ and $N\in\G_{\infty\infty}$.
There are exact sequences $0\to M^{\oplus ab}\to N^{\oplus b}\to\syz^nM^{\oplus b^2}\to0$ and $0\to M^{\oplus a^2}\to N^{\oplus a}\to\syz^nM^{\oplus ab}\to0$, the former of which induces an exact sequence $0\to\syz^nM^{\oplus ab}\to\syz^nN^{\oplus b}\oplus P_1\to\syz^{2n}M^{\oplus b^2}\to0$ with $P_1$ projective.
We get an exact sequence
$$
0\to M^{\oplus a^2}\to N^{\oplus a}\to\syz^nN^{\oplus b}\oplus P_1\to\syz^{2n}M^{\oplus b^2}\to0.
$$
There are exact sequences $0\to M^{\oplus ab^2}\to N^{\oplus b^2}\to\syz^nM^{\oplus b^3}\to0$ and $0\to M^{\oplus a^3}\to N^{\oplus a^2}\to\syz^nN^{\oplus ab}\oplus P_1^{\oplus a}\to\syz^{2n}M^{\oplus ab^2}\to0$, the former of which induces an exact sequence $0\to\syz^{2n}M^{\oplus ab^2}\to\syz^{2n}N^{\oplus b^2}\oplus P_2\to\syz^{3n}M^{\oplus b^3}\to0$ with $P_2$ projective.
We get an exact sequence
$$
0\to M^{\oplus a^3}\to N^{\oplus a^2}\to\syz^nN^{\oplus ab}\oplus P_1^{\oplus a}\to\syz^{2n}N^{\oplus b^2}\oplus P_2\to\syz^{3n}M^{\oplus b^3}\to0.
$$
Iterating this procedure, we obtain exact sequences
$$
\{0\to M^{\oplus a^i}\to G_{i-1}\to\cdots\to G_1\to G_0\to\syz^{in}M^{\oplus b^i}\to0\}_{i\ge1}
$$
with $G_j\in\G_{\infty\infty}$ for all $0\le j\le i-1$.
Using the fact that $M$ is in $\G_{\infty0}$, we see that the $\Lambda$-duals of these exact sequences are again exact.
Now our Theorem \ref{3} deduces that $M^{\oplus a^i}$ belongs to $\G_{0i}$ for each $i\ge1$.
Since $a>0$, this implies that $M\in\G_{0i}$ for all $i\ge1$, whence $M\in\G_{0\infty}$.
It follows that $M\in\G_{\infty0}\cap\G_{0\infty}=\G_{\infty\infty}$.
Therefore $M$ is totally reflexive, and it has finite Gorenstein dimension.
\end{proof}

\section{Reducing projective dimension and upper reducing projective dimension}\label{11}

In this section, we explore reducing projective dimension and upper reducing projective dimension mainly over a commutative noetherian local ring.
We obtain a close connection of upper reducing projective dimension with complexity.
We also investigate some properties of local rings by means of finiteness of reducing projective dimension and of upper reducing projective dimension.

Modifying the definitions of reducing projective and Gorenstein dimensions, we introduce the following homological invariants.

\begin{dfn}
Let $\Lambda$ be a noetherian ring.
Let $M$ be a finitely generated $\Lambda$-module.
We define the {\em upper reducing projective dimension} $\ured_\Lambda^{\pd}M$ and the {\em upper reducing Gorenstein dimension} $\ured_\Lambda^{\gdim}M$ of $M$ as follows; let $\varpi\in\{\pd,\gdim\}$.
\begin{enumerate}[(1)]
\item
One has $\varpi_\Lambda(M)<\infty$ if and only if $\ured_\Lambda^{\varpi}M=0$.
\item
Assume $\varpi_\Lambda(M)=\infty$.
If there exist exact sequences $\{0\to M_{i-1}\to M_i\to\syz^{n_i}M_{i-1}\to0\}_{i=1}^s$ in $\mod\Lambda$ with $n_i\ge0$, $s>0$, $M_0=M$ and $\varpi_\Lambda(M_s)<\infty$, then $\ured_\Lambda^{\varpi}M$ is the minimum of such $s$.
\item
If no such exact sequences exist, then $\ured_\Lambda^{\varpi}M=\infty$.
\end{enumerate}
Note that there is an inequality $\ured_\Lambda^{\varpi}M\ge\red_\Lambda^{\varpi}M$, which is the reason for the names of {\em upper} reducing projective/Gorenstein dimensions.
\end{dfn}

Similarly to Remark \ref{8}, the statements below hold.

\begin{rem}\label{10}
Let $(R,\m,k)$ be a commutative noetherian local ring.
One has the following equivalences.
\begin{enumerate}[(1)]
\item
$\text{$R$ is regular}\iff
\text{$\ured_R^{\pd}M=0$ for all $M\in\mod R$}\iff
\text{$\ured_R^{\pd}k=0$.}$
\item
$\text{$R$ is Gorenstein}\iff
\text{$\ured_R^{\gdim}M=0$ for all $M\in\mod R$}\iff
\text{$\ured_R^{\gdim}k=0$.}$
\end{enumerate}
\end{rem}

The following lemma is the upper reducing homological dimension version of Lemma \ref{redseq}, whose proof is similar and so omitted.

\begin{lem}\label{redseq2}
Let $\Lambda$ be a noetherian ring, and $M$ a finitely generated $\Lambda$-module.
Let $\varpi\in\{\pd,\gdim\}$, and let $r$ be a positive integer.
If one has $\ured_\Lambda^\varpi M\leq r$, then there exist exact sequences $\{0\to M_{i-1}\to M_i\to\syz^{n_i}M_{i-1}\to0\}_{i=1}^r$ in $\mod\Lambda$ such that $n_i\ge0$, $M_0=M$ and $\varpi_\Lambda(M_r)<\infty$.
\end{lem}

Here we recall the definition of complexity, and that of reducible complexity introduced by Bergh \cite{B}.

\begin{dfn}\label{25}
Let $(R,\m,k)$ be a commutative noetherian local ring.
\begin{enumerate}[(1)]
\item
Let $M$ be a finitely generated $R$-module.
The {\em complexity} $\cx_RM$ of $M$ is defined to be the infimum of integers $n\ge0$ such that there exists a real number $\alpha$ with $\beta_i^R(M)\le\alpha i^{n-1}$ for all $i\gg0$.
Here $\beta_i^R(M)$ stands for the $i$th Betti number of the $R$-module $M$, i.e., $\beta_i^R(M)=\dim_k\Tor_i^R(M,k)$.
\item
We define the full subcategory $\X$ of $\mod R$ inductively as follows.
Let $M\in\mod R$.
\begin{enumerate}[(a)]
\item
If $\pd_RM<\infty$ (i.e., $\cx_RM=0$), then $M\in\X$.
\item
If $0<\cx_RM<\infty$ and there exists an exact sequence $0\to M\to N\to\syz^nM\to0$ in $\mod R$ with $n\ge0$ such that $\cx_RN<\cx_RM$, $\depth_RN=\depth_RM$ and $N\in\X$, then $M\in\X$.
\end{enumerate}
The modules belonging to $\X$ are said to have {\em reducible complexity}.
Note by definition that any module having reducible complexity has finite complexity.
\end{enumerate}
\end{dfn}

\begin{rem}\label{5}
Let $R$ be a commutative noetherian local ring.
By \cite[Proposition 2.2(i)]{B}, if a finitely generated $R$-module has finite complete intersection dimension, then it has reducible complexity.
In particular, if $R$ is a complete intersection, then every finitely generated $R$-module has reducible complexity.
\end{rem}

A reason why we introduce upper reducing projective dimension is that the following theorem holds.

\begin{thm}\label{4}
Let $R$ be a commutative noetherian local ring.
Let $M$ be a finitely generated $R$-module.
Then $\cx_RM\le\ured_R^{\pd}M$.
If $M$ has reducible complexity, then the equality $\cx_RM=\ured_R^{\pd}M$ holds.
\end{thm}

\begin{proof}
First, we prove the inequality $\cx_RM\le\ured_R^{\pd}M$.
We may assume $\ured_R^{\pd}M=r<\infty$.
We use induction on $r$.
If $r=0$, then $\pd_RM<\infty$ and hence $\cx_RM=0$.
Let $r\ge1$.
There exist exact sequences $\{0\to M_{i-1}\to M_i\to\syz^{n_i}M_{i-1}\to0\}_{i=1}^r$ in $\mod R$ such that $M_0=M$ and $\pd_RM_r<\infty$.
Note then that $\ured_R^{\pd}M_1\le r-1$.
The induction hypothesis implies $\cx_RM_1\le r-1$.
Putting $n=n_1$, $N=M_1$ and $c=\cx_RN$, we have an exact sequence $0\to M\to N\to\syz^nM\to0$ and $c\le r-1$.
There exist a real number $\alpha>0$ and an integer $u>0$ such that $\beta_i^R(N)\le\alpha i^{c-1}$ for all $i\ge u$.
The induced exact sequence $\Tor_i^R(N,k)\to\Tor_i^R(\syz^nM,k)\to\Tor_{i-1}^R(M,k)$ for each $i$ yields $\beta_{i+n}^R(M)\le\beta_i^R(N)+\beta_{i-1}^R(M)\le\alpha i^{c-1}+\beta_{i-1}^R(M)$ for all $i\ge u$.
Fix two integers $p,q$ with $p\ge0$ and $u-1\le q\le u-1+n$.
We have
\begin{align}\label{26}
\beta_{p(n+1)+q}^R(M)-\beta_q^R(M)
&=\textstyle\sum_{l=1}^p(\beta_{l(n+1)+q}^R(M)-\beta_{(l-1)(n+1)+q}^R(M))\\
&\notag\le\textstyle\sum_{l=1}^p\alpha((l-1)(n+1)+q+1)^{c-1}\\
&\notag\le\textstyle\sum_{l=1}^p\alpha((p-1)(n+1)+q+1)^{c-1}\\
&\notag=\alpha p((p-1)(n+1)+q+1)^{c-1}
\le\alpha(p(n+1)+q)^c.
\end{align}
Putting $\gamma=\max_{u-1\le h\le u-1+n}\{\beta_h^R(M)\}$, we get $\beta_{p(n+1)+q}^R(M)\le\gamma+\alpha(p(n+1)+q)^c$ for all $p,q$ with $p\ge 0$ and $u-1\le q\le u-1+n$.
Note that $\{p(n+1)+q\mid p\ge0,\,u-1\le q\le u-1+n\}$ coincides with the set of integers at least $u-1$.
Hence $\beta_i^R(M)\le\gamma+\alpha i^c$ for all $i\ge u-1$, which implies $\beta_i^R(M)\le(\alpha+1)i^c$ for all $i\gg0$.
Therefore $\cx_RM\le c+1\le r$ as desired.

Next, we prove the opposite inequality $\cx_RM\ge\ured_R^{\pd}M$, assuming that $M$ has reducible complexity.
Let $c:=\cx_RM<\infty$.
We use induction on $c$.
When $c=0$, the module $M$ has finite projective dimension, and $\ured_R^{\pd}M=0$.
Let $c\ge1$.
Since $M$ has reducible complexity, there exists an exact sequence
\begin{equation}\label{21}
0\to M\to N\to \syz^nM\to 0
\end{equation}
in $\mod R$ such that $n\ge0$, $\cx_RN<c$, $\depth_RN=\depth_RM$ and $N$ has reducible complexity.
We can apply the induction hypothesis to $N$ to get $\ured_R^{\pd}N\le\cx_RN=:e$.
If $e=0$, then $N$ has finite projective dimension, and it follows from \eqref{21} that $\ured_R^{\pd}M\le1\le c$.
Let $e\ge1$.
Lemma \ref{redseq2} implies that there exist a family $\{0\to M_{i-1}\to M_i\to\syz^{n_i}M_{i-1}\to0\}_{i=2}^{e+1}$ of exact sequences in $\mod R$ with $n_i\ge0$, $M_1=N$ and $\pd_RM_{e+1}<\infty$.
Putting $n_1=n$ and $M_0=M$ to include \eqref{21} in this family, we get the family $\{0\to M_{i-1}\to M_i\to\syz^{n_i}M_{i-1}\to0\}_{i=1}^{e+1}$ of exact sequences in $\mod R$ with $n_i\ge0$, $M_0=M$ and $\pd_RM_{e+1}<\infty$.
It follows that $\ured_R^{\pd}M\le e+1\le c$.
\end{proof}

We recall the definition of an AB ring introduced by Huneke and Jorgensen \cite{HJ}.

\begin{dfn}
A commutative noetherian local ring $R$ is called {\em AB} if there exists an integer $n$ such that for all finitely generated $R$-modules $M,N$ with $\Ext^{\gg0}_R(M,N)=0$ one has $\Ext_R^{\ge n}(M,N)=0$.
\end{dfn}

It is asked in \cite[Question 2.2]{AC} whether a commutative noetherian local ring over which every finitely generated module has finite reducing projective dimension is a complete intersection.
We can prove the following result, which shows that the question is affirmative if we replace reducing projective dimension with upper reducing projective dimension, and that such a ring is AB.

\begin{cor}\label{6}
Let $(R,\m,k)$ be a commutative noetherian local ring.
Consider the following conditions.

{\rm(1)} $R$ is a complete intersection.\quad
{\rm(2)} All $M\in\mod R$ have reducible complexity.\quad
{\rm(3)} $\ured_R^{\pd}k<\infty$.

{\rm(4)} $\ured_R^{\pd}M<\infty$ for all $M\in\mod R$.\quad
{\rm(5)} $\red_R^{\pd}M<\infty$ for all $M\in\mod R$.\quad
{\rm(6)} $R$ is AB.\\
Then the implications {\rm(1)} $\Leftrightarrow$ {\rm(2)} $\Leftrightarrow$ {\rm(3)} $\Leftrightarrow$ {\rm(4)} $\Rightarrow$ {\rm(5)} $\Rightarrow$ {\rm(6)} hold.
\end{cor}

\begin{proof}
It is evident that the implications (3) $\Leftarrow$ (4) $\Rightarrow$ (5) hold.
It follows from Remark \ref{5} that (1) implies (2).
Theorem \ref{4} particularly says that (2) implies (4).
If (3) holds, then Theorem \ref{4} implies that $k$ has finite complexity, and (1) follows by \cite[Theorem 8.1.2]{A}.
Thus (3) implies (1).

It remains to show that (5) implies (6).
Put $t=\depth R$.
Take nonzero finitely generated $R$-modules $M,N$ such that $\Ext^{\gg 0}_R(M,N)=0$.
By assumption, $r:=\red_R^{\pd}M$ is finite.
When $r=0$, we have $\pd_RM<\infty$, which implies $\pd_RM\le t$, and $\Ext_R^{>t}(M,N)=0$.
We may assume $r>0$.
Lemma \ref{19} gives exact sequences $\{0\to M_{i-1}^{\oplus a_i}\to M_i\to\syz^{n_i}M_{i-1}^{\oplus b_i}\to0\}_{i=1}^r$ in $\mod R$ with $n_i\ge0$, $a_i,b_i>0$, $M_0=M$ and $\pd_RM_r<\infty$.
By the proof of \cite[5.1]{AC} we have $\sup\{i\in\N\mid\Ext_R^i(M,N)\not=0\}=\sup\{i\in\N\mid\Ext_R^i(M_r,N)\not=0\}=\pd_R M_r\le t$.
Thus $R$ is AB.
\end{proof}

Here is an example of a ring $R$ admitting a module $M$ with $\ured_R^{\pd}M>\red_R^{\pd}M$.
This example also says that the reducing projective dimension version of Theorem \ref{4} does not hold in general.

\begin{ex}
Consider the ring $R=k[x,y]/(x^2,xy,y^2)$ with $k$ a field.
Then $\pd_Rk=\infty$ and there is an exact sequence $0\to k^{\oplus2}\to R\to k\to0$, whence $\red_R^{\pd}k=1$.
On the other hand, as $R$ is not a complete intersection, $\ured_R^{\pd}k=\infty$ by Corollary \ref{6}.
Therefore one has $\ured_R^{\pd}k>\red_R^{\pd}k$.
Moreover, since $\cx_Rk=\infty$ by \cite[Theorem 8.1.2]{A}, one has $\cx_Rk>\red_R^{\pd}k$.
\end{ex}

\section{Upper reducing Gorenstein dimension}\label{13}

In this section, we consider Gorenstein analogues of the results obtained in the previous section.

We recall the definition of plexity in the sense of Avramov \cite{Av}, and introduce Gorenstein versions of complexity and reducible complexity.

\begin{dfn}
Let $(R,\m,k)$ be a commutative noetherian local ring.
Denote by $\modo R$ the full subcategory of $\mod R$ consisting of modules that are locally free on the punctured spectrum of $R$.
\begin{enumerate}[(1)]
\item
The {\em plexity} $\px_RM$ of $M\in\mod R$ is by definition the infimum of integers $n\ge0$ such that there exists a real number $\alpha$ with $\mu_R^i(M)\le\alpha i^{n-1}$ for all $i\gg0$.
Here $\mu_R^i(M)$ denotes the $i$th Bass number of $R$, that is, $\mu_R^i(M)=\dim_k\Ext_R^i(k,M)$.
\item
Let $M\in\modo R$.
Then $\Ext_R^i(M,R)$ has finite length for all $i>0$, and one can define the {\em Gorenstein complexity} $\gcx_RM$ of $M$ as the infimum of integers $n\ge0$ such that there exists a real number $\alpha$ with $\ell_R(\Ext_R^i(M,R))\le\alpha i^{n-1}$ for all $i\gg0$.
\item
We define the full subcategory $\X$ of $\modo R$ inductively as follows.
Let $M\in\modo R$.
\begin{enumerate}[(a)]
\item
If $\gdim_RM<\infty$ (i.e., $\gcx_RM=0$), then $M\in\X$.
\item
If $0<\gcx_RM<\infty$ and there exists an exact sequence $0\to M\to N\to\syz^nM\to0$ in $\mod R$ with $n\ge0$ such that $\gcx_RN<\gcx_RM$, $\depth_RN=\depth_RM$ and $N\in\X$, then $M\in\X$.
\end{enumerate}
The modules belonging to $\X$ are said to have {\em reducible Gorenstein complexity}.
\end{enumerate}
Note that one has the equality $\gcx_Rk=\px_RR$, and that any module having reducible Gorenstein complexity has finite Gorenstein complexity.
\end{dfn}

\begin{rem}\label{27}
Let $(R,\m,k)$ be a commutative noetherian local ring.
\begin{enumerate}[(1)]
\item
If in the definition of Gorenstein complexity we replace the length function $\ell_R$ with the minimal number of generators function $\nu_R$, then $\gcx_RM$ coincides with the {\em complexity} $\cx_R(M,R)$ of the pair of modules $(M,R)$ in the sense of Avramov and Buchweitz \cite{ABu}.
When $R$ is artinian, one actually has $\gcx_RM=\cx_R(M,R)$ for each finitely generated $R$-module $M$ by \cite[Corollary 2.6]{DV}.
\item
One has:
$\text{$R$ is Gorenstein}\iff
\text{$\gcx_RM=0$ for all $M\in\modo R$}\iff
\text{$\gcx_Rk=0$}\iff
\text{$\px_RR=0$.}$
\end{enumerate}
\end{rem}

The following result is a Gorenstein version of Theorem \ref{4}.

\begin{prop}\label{7}
Let $R$ be a commutative noetherian local ring.
Let $M$ be a finitely generated $R$-module that is locally free on the punctured spectrum of $R$.
Then there is an inequality $\gcx_RM\le\ured_R^{\gdim}M$.
If $M$ has reducible Gorenstein complexity, then the equality $\gcx_RM=\ured_R^{\gdim}M$ holds.
\end{prop}

\begin{proof}
The proposition can be shown in a simlar way as in the proof of Theorem \ref{4}.

As to the first assertion of the proposition, we use induction on $r=\ured_R^{\gdim}M\in\N$.
If $r=0$, then $\gdim_RM<\infty$ and $\gcx_RM=0$.
For $r\ge1$, there exist exact sequences $\{0\to M_{i-1}\to M_i\to\syz^{n_i}M_{i-1}\to0\}_{i=1}^r$ in $\mod R$ such that $M_0=M$ and $\gdim_RM_r<\infty$.
Note then that $\ured_R^{\gdim}M_1\le r-1$ and that $M_1$ is locally free on the punctured spectrum of $R$.
The induction hypothesis implies $\gcx_RM_1\le r-1$.
Putting $n=n_1$, $N=M_1$ and $c=\gcx_RN$, we have an exact sequence $0\to M\to N\to\syz^nM\to0$ and $c\le r-1$.
There are a real number $\alpha>0$ and an integer $u>0$ such that $\ell_R(\Ext_R^i(N,R))\le\alpha i^{c-1}$ for all $i\ge u$.
The induced exact sequence $\Ext_R^{i-1}(M,R)\to\Ext_R^i(\syz^nM,R)\to\Ext_R^i(N,R)$ for each $i$ yields $\ell_R(\Ext_R^{i+n}(M,R))\le\ell_R(\Ext_R^{i-1}(M,R))+\alpha i^{c-1}$ for all $i\ge u$.
Making a computation similar to \eqref{26}, we can deduce that $\ell_R(\Ext_R^i(M,R))\le(\alpha+1)i^c$ for all $i\gg0$.
Thus $\gcx_RM\le c+1\le r$ as desired.

As for the second assertion of the proposition, in the proof of the second assertion of Theorem \ref{4}, we replace complexity with Gorenstein complexity, and projective dimension with Gorenstein dimension.
Then we are done.
\end{proof}

The corollary below is an immediate consequence of Proposition \ref{7}.

\begin{cor}\label{9}
Let $R$ be a commutative noetherian local ring.
Suppose that the residue field $k$ has finite upper reducing Gorenstein dimension.
Then $R$ has finite plexity: $\px_RR<\infty$.
\end{cor}

\begin{rem}
According to \cite[\S1]{CSV} and \cite[Question A]{S}, it is an open problem whether a commutative noetherian local ring $(R,\m,k)$ is Gorenstein if $\px_RR<\infty$.
If it turns out to be true, then Corollary \ref{9} and Remarks \ref{10}(2), \ref{27}(2) will imply that $R$ is Gorenstein if and only if $\ured_R^{\gdim}M<\infty$ for all $M\in\mod R$, if and only if $\ured_R^{\gdim}k<\infty$, which is a Gorenstein version of the equivalences (1) $\Leftrightarrow$ (3) $\Leftrightarrow$ (4) in Corollary \ref{6}.
\end{rem}


\end{document}